\documentclass{amsart}
\usepackage{MnSymbol}
\usepackage[colorlinks=false,pagebackref,hyperindex]{hyperref}

\allowdisplaybreaks

\newcommand{\FF}{\mathbf F}
\newcommand{\ZZ}{\mathbf Z}

\newcommand{\qedbox}{\rule{2mm}{2mm}}
\def\qedbox{\hbox{\vbox{\hrule\hbox{\vrule\kern3pt\vbox{\kern6pt}\kern3pt\vrule}\hrule}}}%

\theoremstyle{Theorem}
\newtheorem{thm}{Theorem}[section]
\newtheorem{cor}[thm]{Corollary}
\newtheorem{lem}[thm]{Lemma}

\theoremstyle{definition}

\newtheorem{dff}[thm]{Definition}
\newtheorem{xmp}[thm]{Example}
\newtheorem{rmk}[thm]{Remark}

\def\goth{\mathfrak}
\def\HK{\textnormal{HK}}
\newcommand{\f}[1]{[#1]}

\def\dotdiv{\ \hbox{\.{-\!-\!-\ }}}
\def\dotdiv{\, {\begin{picture}(1,1)(-1,-2)\put(2.5,3){\circle*{1.6}}\linethickness{.3mm}\line(1,0){5}\end{picture}}\ \ }

\title{Hilbert-Kunz functions of $2 \times 2$ determinantal rings}
\author{Lance Edward Miller and Irena Swanson}
\thanks{The first author was partially supported by a National Science Foundation VIGRE Grant,
\#0602219}
\begin{document}

\maketitle 

\begin{abstract}
Let $k$ be an arbitrary field (of arbitrary characteristic) and let $X = [x_{i,j}]$ 
be a generic $m \times n$ matrix of variables. Denote by $I_2(X)$ the ideal in 
$k[X] = k[x_{i,j}: i = 1, \ldots, m; j = 1, \ldots, n]$
generated by the $2 \times 2$ minors of $X$. 
Using Gr\"{o}bner basis
we give a recursive formulation
for the lengths of the $k[X]$-module $k[X]/(I_2(X) + (x_{1,1}^q, \ldots, x_{m,n}^q))$
as $q$ varies over all positive integers.
This is a generalized Hilbert-Kunz function,
and our formulation proves that it is a polynomial function in~$q$.
We apply our method to give closed forms for these Hilbert-Kunz functions for cases $m \le 2$.
\end{abstract}

\section{Introduction}

Let $R$ be a ring of characteristic $p > 0$ and set $q = p^e$. 
For a zero-dimensional ideal $I$ and a finitely generated $R$-module $M$,
the Hilbert-Kunz function $\HK_{M,I}(q)$
is the $R$-module length of $M/I^{\f{q}}M$,
where $I^{\f{q}}$ is the ideal
generated by the $q$-th powers of elements of a generating set of~$I$.
Hilbert-Kunz functions were initially studied by Kunz \cite{Kun76}. 
In contrast with the Hilbert-Samuel function,
which agrees with a polynomial for large input,
the Hilbert-Kunz function is in general not a polynomial function
even asymptotically:

\begin{xmp}(c.f.\ \cite{HM83})\label{xmp:notpoly}
If $R = \FF_5[[w,x,y,z]]/(w^4+x^4+y^4+z^4)$,
then the characteristic is~$5$,
and for $e \geq 1$,
$\HK_{R,(x,y,z,e)}(5^e) = {168 \over 61}5^{3e} - {107 \over 61}3^e$.
\end{xmp}

Monsky \cite[Theorem 3.10]{Mon83}
showed that $\HK_{M,I}(q) = cq^d + O(q^{d-1})$
for a real constant $c$ and where $d = \dim M$.
The real number $c$ is called the Hilbert-Kunz multiplicity of $M$
with respect to $I$
and is denoted $e_{\HK}(I;M)$.
There is a great deal of computational evidence \cite{Mon08,Mon09}
that the Hilbert-Kunz multiplicities
can take on non-rational algebraic or even transcendental values,
but no such examples have been definitively established.
The Hilbert-Kunz function has proven difficult to compute in any generality, though there are computations done for explicit examples in small dimensions or special cases \cite{Con96,FT03,Kr07,Mon,MT06,Tei02,Wa}.
Deeper coefficients have only recently been proven to exist in special cases \cite{HY,HMM04}. The well-established 
sensitivity of the Hilbert-Kunz function to the singularities of the underlying space serves as a driving motivation for 
their study. For more details, on Hilbert-Kunz theory see \cite{Hun96} or \cite[Section 8.4]{ST}.

\smallskip

The main subject of this article is to compute the {\bf (generalized)
Hilbert-Kunz function} of determinantal varieties of size $2$ minors.
The setting is the quotient of a polynomial ring
in $m \cdot n$ variables $x_{i,j}$ over an arbitrary field $k$
(of arbitrary characteristic)
modulo the ideal $I_2(X)$ generated by the $2 \times 2$ minors
of the generic matrix $[x_{i,j}]$,
and we study the length function
$$
\HK_{k[X]/I_2(X),\goth{m}}(q)
= \lambda\left(\frac{k[X]}{I_2(X)+\goth{m}^{\f{q}}}\right),
$$
where $q$ varies over all non-negative integers,
$\goth{m}^{\f{q}} = (x_{i,j}^q: i, j)$, and $\lambda$ stands for length.
We will simply refer to this function as the Hilbert-Kunz function,
and the corresponding (generalized) Hilbert-Kunz multiplicity 
of $k[X]/I_2(X)$ with respect to $\goth{m}$
equals
$$
e_{\HK} (k[X]/I_2(X);\goth{m})
= \lim_{q \to \infty}
\frac{\HK_{k[X]/I_2(X), \goth{m}}(q)}{q^{m+n-1}}.
$$
We give a recursive formulation of the Hilbert-Kunz function
which enables us to prove 
that it 
is a polynomial in $q$ (Corollary~\ref{cor:poly}).
We give closed forms for the Hilbert-Kunz function
and the Hilbert-Kunz multiplicity
when the generic matrix is of size $2 \times n$
(Theorem~\ref{thm:2byn},
Corollary~\ref{corHKmult}).
Both are independent of the characteristic of the field. 
We note that our recursive formulation gives lengths of various
other ideals as well (see discussion below Definition~\ref{def:genbasis}).

Results on the Hilbert-Kunz multiplicity of $k[X]/I_2(X)$
at the maximal ideal of variables are well understood 
by the work of Buchweitz, Chen, Pardue \cite{BCP97},
Eto \cite{Eto02},
Eto and Yoshida \cite{EY03},
Watanabe~\cite{Wa},
and Watanabe and Yoshida \cite{WY04}.
The first calculation had an integral form \cite{BCP97},
which was later put into more combinatorial form by Eto and Yoshida using Sterling numbers of the second kind \cite{EY03,WY04}:
The Hilbert-Kunz multiplicity of $k[X]/I_2(X)$ with respect to $\goth{m}$ is
$$e_{\HK}(k[X]/I_2(X); \goth{m}) = {n! \over d!}S(d,n) - {1 \over d!} \sum_{r = 1}^{m - 1} \sum_{s=1}^{m-r} \binom{m}{r+s}\binom{n}{s}(-1)^{m+r}s^d.$$
The central technique in the results of Eto and Yoshida is viewing the determinantal ring
of $2 \times 2$ minors as a Segre product. They work with the relevant lengths by counting monomials in each factor similar to the monomials counted in this article (See Remark~\ref{rmk:EY}).
Their work only yields the Hilbert-Kunz multiplicity,
and the approach of this article is to use Gr\"obner bases to compute the lengths. 
Our approach has the advantage of giving a recursive method which can be used to calculate not only the multiplicity, but the complete Hilbert-Kunz function.

\medskip
The authors thank Karl Schwede for careful readings of the document.
The motivating question for this article was asked at a problem session in a recent AIM workshop entitled ``Test ideals and multiplier ideals" organized by Karl Schwede and Kevin Tucker. 
The authors thank Kevin Tucker, Karl Schwede, Anurag Singh, and Marcus Robinson for many helpful discussions and input.

\section{Gr\"obner bases}\label{sec:GB}

We use the settings from the introduction in the rest of the paper. We 
impose any diagonal order on monomials, that is,
any monomial order in which the leading term
of the determinant of any $2 \times 2$ submatrix
is the product of the two diagonal entries.
Examples of such diagonal orders are the lexicographic order
with variables themselves ordered lexicographically by their indices
$x_{1,1} > x_{1,2} > \cdots > x_{1,n} > x_{2,1} > \cdots > x_{m,n-1} > x_{m,n}$,
or the degree reverse lexicographic orders
with variables ordered instead along the rows from right to left in each row
from the top row to the bottom row.

\begin{dff}\label{defstaircase}
We call a monomial $\prod_{i,j} x_{i,j}^{p_{i,j}}$
a {\bf staircase monomial}
if whenever $i < i'$ and $j < j'$,
then $p_{i,j} p_{i',j'} = 0$.
Thus the indices $(i,j)$ for which $p_{i,j} \not = 0$
lie on a southwest-northeast staircase in the two-dimensional integer lattice,
such as in the following matrix:
$$\left[
\begin{matrix}
 & & & & & \bullet & \bullet & \bullet \cr
 & & & & & \bullet & \cr
 & & & \bullet & \bullet & \bullet \cr
 & & & \bullet & \cr
 \bullet & \bullet & \bullet & \bullet \cr
\end{matrix}\right]
$$
We call a staircase monomial a {\bf stair monomial}
if there exist $c \in \{1, \ldots, m\}$ and $d \in \{1, \ldots, n\}$
such that $p_{l,k} = 0$ whenever $(l-c)(k-d) \not = 0$.
Thus the indices $(i,j)$ for which $p_{i,j} \not = 0$
all lie in the union of part of row $c$ with part of column $d$,
either in a $\lefthalfcap$ or a $\righthalfcup$ configuration.
A stair monomial
is called a {\bf $\bf q$-stair monomial}
if for such $c, d$,
$\sum_k p_{c,k} = q = \sum_k p_{k,d}$.
\end{dff}

\begin{thm} \label{thm:GB}
Let $G$ be the set of all $q$-stair monomials
and all $2 \times 2$ determinants of $X$.
Then $G$ is a minimal reduced Gr\"obner basis for $I_2(X) + \goth{m}^{\f{q}}$
in any diagonal order. 
\end{thm}

\begin{proof}
Set $I = I_2(X) + \goth{m}^{\f{q}}$. First we prove that $G \subseteq I$. 
We only have to prove that an arbitrary $q$-stair monomial
$\beta = \prod x_{i,j}^{p_{i,j}}$ is in $I$. We proceed by induction on the degree of $\beta$. 
The smallest possible degree of such $\beta$ is $q$. 
In this case $\beta = x_{c,d}^q \in \goth{m}^{\f{q}} \subseteq I$.
Now suppose that the degree of $\beta$ is strictly greater than $q$.
By definition of $q$-stair monomials,
there exist $i < i'$ and $j < j'$ such that $p_{i,j'} p_{i',j} \not = 0$.
Set $\alpha = \beta/(x_{i,j'}^{p_{i,j'}} x_{i',j}^{p_{i',j}}
x_{i,j}^{p_{i,j}} x_{i',j'}^{p_{i',j'}})$
and $m = \min\{p_{i,j'}, p_{i',j}\}$.
Adding $$\alpha x_{i,j'}^{p_{i,j'}-m} x_{i',j}^{p_{i',j}-m} x_{i,j}^{p_{i,j}} x_{i',j'}^{p_{i',j'}}
(x_{i,j}^m x_{i',j'}^m - x_{i',j}^m x_{i,j'}^m) \in I$$
to $\beta$,
we get the monomial
$\beta' = \alpha x_{i,j'}^{p_{i,j'}-m} x_{i',j}^{p_{i',j}-m}
x_{i,j}^{p_{i,j}+m} x_{i',j'}^{p_{i',j'}+m}$.
By the shape of $q$-stair monomials,
either $\beta'/x_{i,j}^{p_{i,j}+m}$ or $\beta'/x_{i',j'}^{p_{i',j'}+m}$
is another $q$-stair monomial of strictly smaller degree than $\beta$.
By induction on degree
this monomial and hence $\beta$ are in $I$.

Now we prove that $G$ forms a Gr\"obner basis
by going through the Buchberger algorithm.
We need to show that for all $f, g \in G$,
their S-polynomial $S(f,g)$ reduces to $0$ with respect to $G$.
If $f$ and $g$ are both determinants,
then $S(f,g)$ reduces to $0$ with respect to
other determinantal elements of $G$ by \cite{CGG90}.
If both $f$ and $g$ are monomials,
then $S(f,g) = 0$.
So we may assume that $f$ is a determinant and $g$ is a $q$-stair monomial.
If the leading terms of $f$ and $g$ have no variables in common,
then $S(f, g)$ trivially reduces to $0$ with respect to $\{f,g\}$.
So we may assume that the leading terms of $f$ and $g$
have a variable in common.
By the shape southwest-northeast structure of $q$-stair monomials
and the northwest-southeast structure of the leading monomials of determinants,
the leading terms of $f$ and $g$ have precisely one variable in common.
Let $f = x_{a,b} x_{a',b'} - x_{a',b} x_{a,b'}$
with $a < a'$, $b < b'$.


Assume that $g = \prod x_{i,j}^{p_{i,j}}$ is in the configuration $\lefthalfcap$ in row~$c$ and column~$d$.
First suppose that $p_{a,b} \not = 0$.
By direct calculation $S(f,g) = g x_{a,b'} x_{a',b}/x_{a,b}$.
If $a = c$ and $b = d$,
then $g /x_{a,b}$ is a $(q-1)$-stair monomial.
Thus $g x_{a,b'} x_{a',b}/x_{a,b}$ is a $q$-stair monomial
in the $\lefthalfcap$ configuration, 
and so $S(f,g)$ reduces to $0$ with respect to $G$.
If $a = c$ and $b \not = d$,
then $g x_{a,b'} /x_{a,b}$ is a $q$-stair monomial,
so also in this case $S(f,g)$ reduces to $0$ with respect to $G$.
The remaining case, $a \not = c$ and $b = d$, under the assumption $p_{a,b} \not = 0$,
is handled similarly.
Now suppose that $p_{a',b'} \not = 0$.
Again, by direct calculation $S(f,g) = g x_{a,b'} x_{a',b}/x_{a',b'}$.
Necessarily $a' = c$ or $b' = d$.
By symmetry of the diagonal orders without loss of generality $a' = c$.
If in addition $b' = d$,
then $g /x_{a',b'}$ is a $(q-1)$-stair monomial.
If $g = x_{c,d}^q$,
then
$S(f,g) = g x_{a,b'} x_{a',b}/x_{a',b'}$ is a $q$-stair monomial
(in $\righthalfcup$ configuration),
and otherwise (still with $a' = c$, $b' = d$) there exist $j > d$ and $i > c$
such that $p_{c,j} p_{i,d} \not = 0$.
In this case,
$x_{a,b'} x_{c,j}$ reduces with respect to $G$ to
$x_{a,j} x_{c,b'} = x_{a,j} x_{c,d}$
and $x_{a',b} x_{i,d}$ reduces to
$x_{c,d} x_{i,b}$,
so that $S(f,g)$ reduces to a multiple of the $q$-stair monomial
$g x_{c,d}^2 /(x_{a',b'} x_{c,j} x_{i,d})$
$= g x_{c,d} /(x_{c,j} x_{i,d})$. 
Thus it reduces to $0$ with respect to $G$.
Now suppose that $a' = c$ and $b' \not = d$.
Necessarily all exponents in $g$ are strictly smaller than $q$,
and since $g$ is a $q$-stair monomial,
there exists $i > c$ such that $p_{i,d} \not = 0$.
Since $x_{a',b} x_{i,d}$ reduces to $x_{c,d} x_{i,b}$,
it follows that $S(f,g)$ is a multiple of a $q$-stair monomial,
and hence reduces with respect to $G$ to $0$.
Similar reasoning shows that
$S(f,g)$ reduces with respect to $G$ to $0$
also in the case when $g$ is a $q$-stair monomial
in the $\righthalfcup$ configuration.
Thus $G$ is a Gr\"obner basis.
\end{proof}

\begin{rmk}\label{rmk:GrRevLex}
While the Gr\"obner basis constructed in Theorem~\ref{thm:GB}
is a Gr\"obner basis for many orders,
it is not a universal Gr\"obner basis
as it is not a Gr\"obner basis for any antidiagonal order,
and in particular it is not a Gr\"obner basis for the
graded reverse lexicographic order
while keeping the same order on the variables.
However, a completely analogous Gr\"obner basis
can be constructed in that case,
and instead of the $q$-stair monomials
in the southwest-northeast configuration
we need analogous monomials in the southeast-northwest configuration.
\end{rmk}

By standard Gr\"obner basis arguments,
the length of $k[X]/(I_2(X) + \goth{m}^{[q]})$,
namely the Hilbert-Kunz function $\HK_{k[X]/I_2(X),\goth{m}}(q)$ at $q$,
equals the number of monomials in $k[X]$
that are not divisible by the leading term of any element of the Gr\"obner basis $G$
of the ideal $I_2(X) + \goth{m}^{[q]}$.
We proceed to make this collection of monomials explicit.


%
%

\begin{thm}\label{thm:basis}
A $k$-vector space basis  for $k[X]/(I_2 + \goth{m}^{\f{q}})$
consists of staircase monomials $\prod_{i,j} x_{i,j}^{p_{i,j}}$
such that either for all $i  = 1, \ldots, m$, $\sum_j p_{i,j} < q$,
or for all $j  = 1, \ldots, n$,
$\sum_i p_{i,j} < q$.
\end{thm}


\begin{proof}
Theorem~\ref{thm:GB}
gives a Gr\"obner basis $G$ of $I_2 + \goth{m}^{\f{q}}$
in any diagonal monomial order.
Any staircase monomial for which either all row sums or all column sums
are strictly smaller than $q$
is not divisible by the leading term of any element of $G$.
Conversely,
let $M$ be a monomial in $k[X]$
that is not divisible by the leading term of any element of $G$.
Since $G$ contains all $2 \times 2$ determinants
whose leading monomials are products of two variables
in the northwest-southeast configuration,
necessarily $M$ is a staircase monomial.
If some row sum and some column sum of the exponents in~$M$ are at least $q$,
then $M$ is a multiple of a $q$-stair monomial.
Hence it is divisible by an element of~$G$,
which is a contradiction.
So the set of all staircase monomials
for which either all row sums or all column sums are strictly smaller than $q$
equals the set of all monomials in $k[X]$
that are not divisible by the leading terms of any element of $G$.
It follows by the standard Gr\"obner basis arguments
that this set is a $k$-vector space basis.
\end{proof}


\begin{rmk}\label{rmk:EY}
Eto and Yoshida \cite{Eto02,EY03}
used similar vector space methods to compute the Hilbert-Kunz
multiplicity of the ring $k[X]/I_2(X)$
(but not the Hilbert-Kunz function). We now give a translation of their approach and ours. 
Let $d \in \ZZ_+$. Viewing $k[X]/I_2(X)$ as a Segre product $k[z_1,\ldots,z_m] \# k[y_1,\ldots,y_n]$, set $\alpha_{m,d}$ to be 
the number of monomials in $k[z_1, \ldots, z_m]$ of total degree $d$ and $\alpha_{m,d,q}$ to be the number 
of monomials in $k[z_1, \ldots, z_m]$ of total degree $d$ and where $\deg z_i < q$ for all $i$,
and similarly for $\alpha_{n,d}$ and $\alpha_{n,d,q}$ in $k[y_1,\ldots,y_n]$.
Both in \cite{Eto02} and \cite{EY03},
$\lambda_k(k[X]/(I_2(X) + \goth{m}^{q}))$
is given by counting all monomials in $k[z_1,\ldots,z_m] \# k[y_1,\ldots,y_n]$ 
which have total degree $d$ but with the conditions that either all the $z_i$ have degree at most $q$ 
or all the $y_i$ have degree at most $q$
(see \cite[pg. 319]{Eto02}): 
\begin{equation}\label{eq:1}
\lambda_k(k[X]/(I_2(X) + \goth{m}^{\f{q}}))
= \sum_{d = 0}^{(q-1)n} \alpha_{m,d} \alpha_{n,d,q} + \sum_{d = 0}^{(q-1)m} \alpha_{n,d} \alpha_{m,d,q} - \sum_{d = 0}^{(q-1)m} \alpha_{n,d,q} \alpha_{m,d,q}.
\end{equation} 
This formulation is the same as our count of monomials in Theorem~\ref{thm:basis} as we show next.
Consider the usual map $k[X] \to k[z_1,\ldots,z_m] \# k[y_1,\ldots,y_n]$
sending $x_{i,j}$ to $z_i y_j$, 
and thus $\prod_{i,j} x_{i,j}^{p_{i,j}}$ to
$$
\prod_{i,j} (z_i y_j)^{p_{i,j}}
= (\prod_i z_i^{\sum_j p_{i,j}}) (\prod_j y_j^{\sum_i p_{i,j}}).
$$
The exponents in the variables $z_i$ are row sums of
the matrix of the exponents of the monomial in $k[X]$,
and the exponents in the $y_j$ are the column sums.
This matches the basis described in 
Theorem~\ref{thm:basis}.

\end{rmk}

\section{Recursion for computing lengths}

In this section we describe a recursion
that we then use to compute the Hilbert function
$\HK_{k[X]/I_2(X), \goth{m}}(q) = \lambda(k[X]/(I_2(X) + \goth{m}^{\f{q}}))$.
We begin by setting up a useful notation for counting elements of the basis
according to restrictions on the row and column sums. 

\begin{dff}
\label{def:genbasis}
Let $m, n \in \mathbf{N}$,
and let $r_1, \ldots, r_m, c_1, \ldots, c_m \in \mathbf{Z} \cup \{\infty\}$.
Let $N_q(m,n;r_1, \ldots, r_m; c_1, \ldots, c_n)$
be the number of staircase monomials
$\prod_{i,j} x_{i,j}^{p_{i,j}}$
such that
\begin{enumerate}
\item
For all $i = 1, \ldots, m$,
$\sum_j p_{i,j} \le r_i$,
and for all $j = 1, \ldots, n$,
$\sum_i p_{i,j} \le c_j$.
\item
Either
for all $i = 1, \ldots, m$,
$\sum_j p_{i,j} < q$,
or for all $j = 1, \ldots, n$,
$\sum_i p_{i,j} < q$.
\end{enumerate}
Note that $N_q(m,n;r_1,\ldots,r_m; c_1,\ldots,c_n) = 
N_q(n,m;c_1,\ldots,c_n;r_1,\ldots,r_m)$.
If any $r_i$ or $c_j$ is negative,
then $N_q(m,n;r_1,\ldots,r_m; c_1,\ldots,c_n) = 0$,
and if $c_1, \ldots, c_n \ge 0$,
then
$N_q(0,n;;c_1, \ldots, c_n) = 1$. 
\end{dff}

Throughout we abbreviate by writing $\overline{c}$ when $c$ gets repeated.
For example, the symbol
$N_q(m,n; \overline\infty; \overline\infty)$
is an abbreviation of
$N_q(m,n; \infty,\ldots,\infty; \infty,\ldots,\infty)$.

We note that the numbers introduced in the definition above
also compute co-lengths of certain very natural ideals. In particular, 
$N_q(m,n;r_1, \ldots, r_m; c_1, \ldots, c_n)$
equals the length
$$
\lambda \left(
{k[X] \over
I_2(X) + (x_{i,j}^q: i, j) + \sum_{i=1}^m (x_{i,1}, \ldots, x_{i,n})^{r_i + 1}
+ \sum_{j=1}^n (x_{1,j}, \ldots, x_{m,j})^{c_j + 1}
}\right),
$$
where for an ideal $I$,
we set $I^{\infty}$ to be the $0$ ideal.
In particular the Hilbert-Kunz function
$\lambda_k(k[X]/(I_2(X) + \goth{m}^{\f{q}})$
that we are interested in computing
is
$N_q(m,n; \infty,\ldots,\infty; \infty,\ldots,\infty)$
$= N_q(m,n; \overline\infty; \overline\infty)$.
Despite this being our primary interest,
the recursion forces us to consider $r_i$ and $c_i$ different from $\infty$.

We first establish the base case $m = 1$ of the induction. 

\begin{thm}
\label{thmN1}
\begin{enumerate}
\item
$N_q(1,n;\infty;c_1, \ldots, c_n) = \prod_{i=1}^n \min\{c_i+1, q\}$.
\item
$N_q(1,n;\infty;\overline\infty)
= N_q(1,n;\infty;\infty, \ldots, \infty) = q^n$.
\item
If $r < \infty$,
then
$$
N_q(1,n;r; c_1, \ldots, c_n)
= \sum_{i_1=0}^{\min\{c_1,r,q-1\}}
\sum_{i_2=0}^{\min\{c_2,r-i_1,q-1\}}
\cdots
\sum_{i_{n-1}=0}^{\min\{c_{n-1},r-\sum_{j=1}^{n-2} i_j,q-1\}}
\min\{r+1-\sum_{j=1}^{n-1} i_j,c_n+1, q\}.
$$
\item
If $r < q$,
and if $r \le c_1, \ldots, c_n$,
then
$$
N_q(1,n;r; c_1, \ldots, c_n)
= {r + n \choose n}.
$$
In particular,
$N_q(1,n;r; c_1, \ldots, c_n)$
is independent of $c_1, \ldots, c_n$.
(Note that this is the number of ordered partitions of $0,1,\ldots,r$ into $n$
or fewer parts.)
\end{enumerate}
\end{thm}

\begin{proof}
If $m = 1 = n$,
then clearly
$N_q(1,1;r; c) = \min\{r + 1,c + 1,q\}$.
If $n > 1$,
$N_q(1,n;r; c_1, \ldots, c_n)$
counts the number of monomials $\prod_j x_{1,j}^{p_{1,j}}$
where $p_{1,1}$ varies in the set $\{0, \ldots, \min\{c_1, r, q-1\}\}$,
and for each such $p_{1,1}$,
the possible number of rest of the $1 \times (n-1)$ matrix of $p_{1,j}$
has the count of $N_q(1,n-1;r-p_{1,1};c_2, \ldots,c_n)$.
Thus, by repeating this reasoning,
\begin{align*}
N_q&(1,n;r; c_1, \ldots, c_n)
= 
\sum_{i_1=0}^{\min\{c_1,r,q-1\}} N_q(1,n-1;r-i_1;c_2, \ldots,c_n) \cr
&= 
\sum_{i_1=0}^{\min\{c_1,r,q-1\}} \sum_{i_2=0}^{\min\{c_2,r-i_1,q-1\}}
N_q(1,n-2;r-i_1-i_2;c_3, \ldots,c_n) \cr
&= \qquad \vdots \cr
&= \sum_{i_1=0}^{\min\{c_1,r,q-1\}}
\sum_{i_2=0}^{\min\{c_2,r-i_1,q-1\}}
\cdots
\sum_{i_{n-1}=0}^{\min\{c_{n-1},r-\sum_{j=1}^{n-2} i_j,q-1\}}
N_q(1,1;r-\sum_{j=1}^{n-1} i_j;c_n) \cr
&= \sum_{i_1=0}^{\min\{c_1,r,q-1\}}
\sum_{i_2=0}^{\min\{c_2,r-i_1,q-1\}}
\cdots
\sum_{i_{n-1}=0}^{\min\{c_{n-1},r-\sum_{j=1}^{n-2} i_j,q-1\}}
\min\{r+1-\sum_{j=1}^{n-1} i_j,c_n+1, q\}. \cr
\end{align*}
It remains to prove (4).
If $n = 1$,
then
$N_q(1,1;r;c_1) = \min\{r+1,c_1+1,q\}$,
which by assumption equals $r+1 = {r+1 \choose 1}$.
Now suppose that we know the result for $n-1$.
Then
$$
N_q(1,n;r; c_1, \ldots, c_n)
= \sum_{i_1=0}^r {r-i_1 + n-1 \choose n-1}
= {r+n-1+1 \choose n-1+1}
= {r+n \choose n},
$$
which proves the proposition.
\end{proof}

Recall that $N_q(0,n;;c_1, \ldots, c_n) = 1$.
In the next theorem and beyond we utilize the so called `monus' operation defined as 
$a \dotdiv b = \max\{ a -b, 0\}$.

\begin{thm}\label{thmrecur}
For all $m, n \ge 1$
and all $r_i, c_j \in \mathbf{N} \cup \{\infty\}$,
\begin{align*}
N_q&(m,n;r_1,\ldots, r_m; c_1, \ldots, c_n)
= 
N_q(m,n-1;r_1,\ldots, r_m; c_2, \ldots, c_n) \cr
&\hskip2em
+ \sum_{i=1}^{m-1} \sum_{j=1}^{\min\{r_i,q-1\}}
(N_q(m-i,1;r_{i+1}, \ldots, r_m; c_1-j)
\dotdiv N_q(m-i,1;r_{i+1}, \ldots, r_m; q-1-j))\cr
&\hskip6em
\cdot N_q(i,n-1;\min\{r_1,q-1\}, \ldots, \min\{r_{i-1},q-1\}, \min\{r_i,q-1\} - j;c_2,\ldots, c_n) \cr
&\hskip2em
+ \sum_{i=1}^m \sum_{j=1}^{\min\{r_i,q-1\}}
N_q(m-i,1;r_{i+1}, \ldots, r_m; \min\{c_1,q-1\}-j) \cr
&\hskip6em
\cdot (N_q(i,n-1;r_1, \ldots, r_{i-1}, r_i - j;\min\{c_2,q-1\},\ldots, \min\{c_n,q-1\}) \cr
&\hskip9em
\dotdiv
N_q(i,n-1;r_1, \ldots, r_{i-1}, q-1 - j;\min\{c_2,q-1\},\ldots, \min\{c_n,q-1\})) \cr
&\hskip2em
+ \sum_{i=1}^m \sum_{j=1}^{\min\{r_i,q-1\}}
N_q(m-i,1;r_{i+1}, \ldots, r_m; \min\{c_1,q-1\}-j) \cr
&\hskip6em
\cdot N_q(i,n-1;r_1, \ldots, r_{i-1}, \min\{r_i,q-1\} - j;c_2,\ldots, c_n). \cr
\end{align*}
\end{thm}

\begin{proof}
By Theorem~\ref{thm:basis},
it suffices to recursively count all monomials $\prod_{i,j} x_{i,j}^{p_{i,j}}$
as in Definition~\ref{def:genbasis}.
The number of such $[p_{i,j}]$ with all $p_{i,1}$ zero
is $N_q(m,n-1;r_1,\ldots, r_m; c_2, \ldots, c_n)$.
Now suppose that some $p_{i,1}$ is non-zero.
Let $i \in \{1, \ldots, m\}$ be smallest with this property.
By the staircase condition,
there are no non-zero entries in $[p_{i,j}]$ in the submatrix
of rows $i+1, \ldots, m$ and columns $2, \ldots, n$,
and by the assumption on $i$,
there are no non-zero entries in $[p_{i,j}]$ in the first column
in rows $1,\ldots, i-1$.
So it remains to count the possible combinations
of how to fill in the submatrix of the first column in rows $i, \ldots, m$
and the submatrix of rows $1, \ldots, i$
and columns $2, \ldots, n$.
If we fill the first column so that the total sum is $q$ or larger,
then we have to make sure that all the rows in the rest of $[p_{i,j}]$ add
up to strictly less than $q$;
if the first column adds up to strictly less than $q$
and the $i$th row adds up to $q$ or more,
than we need to control all the columns of $[p_{i,j}]$ to be strictly less than $q$;
and finally,
if the first column and the $i$th row each add up to at most $q-1$,
then we have no further restriction on the rest.
This is expressed precisely by the sums in the recursive formulation.
\end{proof}

Theorems~\ref{thmN1} and \ref{thmrecur} immediately give:

\begin{cor} \label{cor:poly}
The Hilbert-Kunz function
$\HK_{k[X]/I_2(X),\goth{m}}(q)$
is a polynomial in $q$.
\hfill\qedbox
\end{cor}

\section{Computing the $2 \times n$ case}\label{sec:2byn}

Our main interest
is the Hilbert-Kunz function $N_q(2,n;\overline{\infty} ;\overline{\infty} )$,
but the recursion in Theorem~\ref{thmrecur}
forces us to calculate also
$N_q(2,n;\overline \infty;\overline {q-1})$,
$N_q(2,n;\infty, r;\overline \infty)$,
$N_q(2,n;\infty, r;\overline {q-1})$
for all $r < q$. 
We begin this section with a summary of the main results proved in this section,
where $r$ in the table is always strictly smaller than $q$:
\smallskip

\begin{center}
\renewcommand{\arraystretch}{2}
\begin{tabular}{|l|l|}
\hline
Theorem &  Result ($r < q$)\\
\hline 
\hline 
Theorem~\ref{thm2ninfq-1} & $N_q(2,n;\overline{\infty}; \overline{q-1})
= q^{n+1} + (n-2) q^{n-1}{q \choose 2}$ \\
\hline 
Theorem~\ref{thm2ninfrinf} & 
$N_q(2,n;\infty,r; \overline{\infty})
= (n-1) {r+n \choose n+1} + (r+1) q^n$ \\
\hline
Theorem~\ref{thm2ninfrq-1} & 
$N_q(2,n;\infty,r; \overline{q-1})
= (r+1) q^n - {r + n \choose n+1}$\\
\hline
Theorem~\ref{thm:2byn} &
$N_q(2,n;\overline{\infty}; \overline{\infty})
= {nq^{n+1} - (n-2)q^n \over 2} +n{n+ q-1 \choose n+1}$ \\
\hline
Corollary~\ref{corHKmult} & 

$\lim_{q \to \infty} {N_q(2,n;\overline \infty; \overline \infty) \over q^{n+1}}
= {n \over 2} + {n \over (n+1)!}$\\
\hline
Theorem~\ref{thmNFsign} &
$N_q(2,n;\infty,r; \infty, \overline{q-1}) = q^n (r+1)$ \\
\hline
Theorem~\ref{thm2q-1rq-1} & 
$N_q(2,n;q-1,r; \overline{q-1})
= q{r+n \choose n}
+ \sum_{i=1}^{n-1} {q+i-1 \choose i+1} {r+n-i \choose n-i}
- n {r+n \choose n+1}$ \\
\hline
Theorem~\ref{thm2q-1rinfty} & 
$N_q(2,n;q-1,r; \overline{\infty})
= q {r+n \choose n}
+ \sum_{i=1}^{n-1} {q+i-1 \choose i+1} {r+n-i \choose n-i}$ \\
\hline
\end{tabular}
\end{center}

\medskip

Throughout we will make use of many detailed but easy lemmas concerning binomial sums. To improve readability the proofs of these can be found in Appendix~\ref{sec:binform}.


We now handle each case required by the recursion in turn. 

\begin{thm}\label{thm2ninfq-1}
For $n \ge 2$,
$N_q(2,n;\overline{\infty}; \overline{q-1})
= q^{n+1} + (n-2) q^{n-1}{q \choose 2}$.
\end{thm}

\proof
By Theorems~\ref{thmN1} and \ref{thmrecur},
\begin{align*}
N_q&(2,n;\overline{\infty}; \overline{q-1})
= N_q(2,n-1;\overline \infty; \overline{q-1}) \cr
&\hskip2em
+ \sum_{j=1}^{q-1}
(N_q(1,1;\infty; q-1-j) \dotdiv N_q(1,1;\infty; q-1-j))
\cdot N_q(1,n-1;q-1-j;\overline{q-1}) \cr
&\hskip2em
+ \sum_{i=1}^2 \sum_{j=1}^{q-1}
N_q(2-i,1;\overline \infty; q-1-j)
\cdot (N_q(i,n-1;\overline \infty;\overline{q-1})
\dotdiv
N_q(i,n-1;\overline \infty, q-1 - j;\overline{q-1})) \cr
&\hskip2em
+ \sum_{i=1}^2 \sum_{j=1}^{q-1}
N_q(2-i,1;\overline \infty; q-1-j)
\cdot N_q(i,n-1;\overline{\infty}, q-1 - j;\overline{q-1}) \cr
&= N_q(2,n-1;\overline \infty; \overline{q-1})
+ \sum_{i=1}^2 \sum_{j=1}^{q-1}
N_q(2-i,1;\overline \infty; q-1-j)
\cdot N_q(i,n-1;\overline \infty;\overline{q-1})
\cr
&= N_q(2,n-1;\overline \infty; \overline{q-1})
+ \sum_{j=1}^{q-1}
N_q(1,1;\infty; q-1-j) \cdot N_q(1,n-1;\infty;\overline{q-1}) \cr
&\hskip2em
+ \sum_{j=1}^{q-1}
N_q(0,1;; q-1-j)
\cdot N_q(2,n-1;\overline \infty;\overline{q-1})
\cr
&= q N_q(2,n-1;\overline \infty; \overline{q-1})
+ \sum_{j=1}^{q-1} (q-j) q^{n-1}
\cr
&= q N_q(2,n-1;\overline \infty; \overline{q-1})
+ q^{n-1} {q \choose 2}.
\cr
\end{align*}
When $n = 2$,
this equals
$q {q+1 \choose 2} + q {q \choose 2} = q^3$,
which is exactly the theorem.
If $n > 2$,
then by induction and reduction above,
\begin{align*}
N_q&(2,n;\overline{\infty}; \overline{q-1})
= q N_q(2,n-1;\overline \infty; \overline{q-1})
+ q^{n-1} {q \choose 2}
\cr
&= q \left(q^n + (n-3) q^{n-2} {q \choose 2} \right)
+ 
q^{n-1} {q \choose 2} \cr
&= q^{n+1} + (n-2) q^{n-1} {q \choose 2}.
&\qedbox
\cr
\end{align*}

\begin{thm}\label{thm2ninfrinf}
For $n \ge 2$ and $r < q$,
$N_q(2,n;\infty,r; \overline{\infty})
= (n-1) {r+n \choose n+1} + (r+1) q^n$.
\end{thm}

\proof
We apply Theorems~\ref{thmN1} and \ref{thmrecur}:
\begin{align*}
N_q&(2,n;\infty,r; \overline{\infty})
= N_q(2,n-1;\infty,r; \overline \infty) \cr
&\hskip2em
+ \sum_{j=1}^{q-1}
(N_q(1,1;r; \infty)
\dotdiv N_q(1,1;r; q-1-j))
\cdot N_q(1,n-1;q-1- j;\overline \infty) \cr
&\hskip2em
+ \sum_{j=1}^{q-1}
N_q(1,1;r; q-1-j)
\cdot (N_q(1,n-1;\infty;\overline {q-1}) \dotdiv
N_q(1,n-1;q-1 - j;\overline{q-1})) \cr
&\hskip2em
+ \sum_{j=1}^r
N_q(0,1;; q-1-j)
\cdot (N_q(2,n-1;\infty, r - j;\overline {q-1})
\dotdiv
N_q(2,n-1;\infty, q-1 - j;\overline{q-1})) \cr
&\hskip2em
+ \sum_{j=1}^{q-1}
N_q(1,1;r; q-1-j) \cdot N_q(1,n-1;q-1-j;\overline \infty)
\cr
&\hskip2em
+ \sum_{j=1}^r
N_q(0,1;; q-1-j)
\cdot N_q(2,n-1;\infty, r- j;\overline \infty)
\cr
&= 
\sum_{j=0}^r N_q(2,n-1;\infty, r- j;\overline \infty)
+ \sum_{j=1}^{q-1}
\left((r+1) - \min\{r+1,q-j\} \right) \cdot {q-1-j+n-1 \choose n-1}
\cr
&\hskip2em
+ \sum_{j=1}^{q-1}
\min\{r+1,q-j\} \cdot \left( q^{n-1} - {q-1-j+n-1 \choose n-1}\right) \cr
&\hskip2em
+ \sum_{j=1}^{q-1} \min\{r+1,q-j\} \cdot {q-1-j+n-1 \choose n-1} \cr
&= 
\sum_{j=0}^r N_q(2,n-1;\infty, r- j;\overline \infty)
+ \sum_{j=1}^{q-1} (r+1) {q-1-j+n-1 \choose n-1}
\cr
&\hskip2em
+ \sum_{j=1}^{q-1}
\min\{r+1,q-j\} \cdot \left( q^{n-1} - {q-1-j+n-1 \choose n-1}\right) \cr
&= 
\sum_{j=0}^r N_q(2,n-1;\infty, r- j;\overline \infty)
+ (r+1) {q+n-2 \choose n}
\cr
&\hskip2em
+ \left(q (r+1) - {r+2 \choose 2} \right) q^{n-1}
- \left((r+1){q + n - 2 \choose n} - {r + n \choose n+1}\right)
\hbox{ (by Lemma~\ref{lmminr+1q-jbinom})}
\cr
&=
\sum_{j=0}^r N_q(2,n-1;\infty, r- j;\overline \infty)
+ (r+1) q^n
- q^{n-1} {r+2 \choose 2}
+ {r+n \choose n+1}.
\cr
\end{align*}
When $n = 2$,
this equals
$\sum_{j=0}^r (r- j+1)q 
+ (r+1) q^2
- q {r+2 \choose 2}
+ {r+2 \choose 3}$
$= {r +2 \choose 2} q 
+ (r+1) q^2
- q {r+2 \choose 2}
+ {r+2 \choose 3}$
$= (r+1) q^2 + {r+2 \choose 3}$,
which is of the desired form.
If $n > 2$,
then by induction and the reduction above,
\begin{align*}
N_q(2,n;\infty,r; \overline{\infty})
&=
\sum_{j=0}^r \left((n-2) {r-j+n-1 \choose n} + (r-j+1) q^{n-1} \right)
\cr
&\hskip1em
+ (r+1) q^n
- q^{n-1} {r+2 \choose 2}
+ {r+n \choose n+1}
\cr
&=
(n-2) {r+n \choose n+1} + {r+2 \choose 2} q^{n-1}
+ (r+1) q^n - q^{n-1} {r+2 \choose 2} + {r+n \choose n+1}
\cr
&=
(n-1) {r+n \choose n+1}
+ (r+1) q^n. &\qedbox
\cr
\end{align*}

\begin{thm}\label{thm2ninfrq-1}
For $n \ge 2$ and $r < q$,
$N_q(2,n;\infty,r; \overline{q-1})
= (r+1) q^n - {r + n \choose n+1}$.
\end{thm}

\proof
By Theorem~\ref{thmrecur},
\begin{align*}
N_q&(2,n;\infty,r; \overline{q-1})
= 
N_q(2,n-1;\infty,r; \overline{q-1}) \cr
&\hskip2em
+ \sum_{j=1}^{q-1}
N_q(1,1;r; q-1-j) \cdot
(N_q(1,n-1;\infty;\overline{q-1})
\dotdiv
N_q(1,n-1;q-1 - j;\overline{q-1})) \cr
&\hskip2em
+ \sum_{j=1}^r
N_q(0,1;; q-1-j)
\cdot (N_q(2,n-1;\infty, r - j;\overline{q-1})
\dotdiv
N_q(2,n-1;\infty, q-1 - j;\overline {q-1})) \cr
&\hskip2em
+ \sum_{j=1}^{q-1}
N_q(1,1;r; q-1-j) \cdot N_q(1,n-1;q-1- j;\overline{q-1}) \cr
&\hskip2em
+ \sum_{j=1}^r
N_q(0,1;; q-1-j)
\cdot N_q(2,n-1;\infty, r- j;\overline{q-1}) \cr
&= 
\sum_{j=0}^r N_q(2,n-1;\infty, r- j;\overline{q-1})
+ \sum_{j=1}^{q-1} \min\{r+1,q-j\} q^{n-1}
\cr
&= 
\sum_{j=0}^r N_q(2,n-1;\infty, r- j;\overline{q-1})
+ (r+1)q^n - {r+2 \choose 2} q^{n-1}
\hbox{ (by Lemma~\ref{lmminr+1q-jbinom})}.
\cr
\end{align*}
When $n = 2$,
this is
\begin{align*}
N_q&(2,2;\infty,r; \overline{q-1})
= 
\sum_{j=0}^r N_q(2,1;\infty, r- j;q-1)
+ (r+1)q^2 - {r+2 \choose 2} q
\cr
&= 
\sum_{j=0}^r \sum_{i=0}^{q-1} \min\{q-i,r-j+1\}
+ (r+1)q^2 - {r+2 \choose 2} q
\cr
&= 
q {r+2 \choose 2} - {r+2 \choose 3}
\hbox{ (by Lemma~\ref{lmothermin}) }
+ (r+1)q^2 - {r+2 \choose 2} q
\cr
&= 
(r+1) q^2
- {r+2 \choose 3},
\cr
\end{align*}
which proves the theorem in case $n = 2$.
Now let $n > 2$.
By above and by induction,
\begin{align*}
N_q&(2,n;\infty,r; \overline{q-1})
= 
\sum_{j=0}^r N_q(2,n-1;\infty, r- j;\overline{q-1})
+ (r+1)q^n - {r+2 \choose 2} q^{n-1}
\cr
&= 
\sum_{j=0}^r
\left(
(r-j+1) q^{n-1}
- {r-j+n-1 \choose n}
\right)
+ (r+1)q^n - {r+2 \choose 2} q^{n-1}
\cr
&= 
{r+2 \choose 2} q^{n-1}
- {r+n \choose n+1}
+ (r+1)q^n - {r+2 \choose 2} q^{n-1}
\cr
&= 
(r+1) q^n - {r+n \choose n+1}. &\qedbox
\cr
\end{align*}

\begin{thm}\label{thm:2byn}
The Hilbert-Kunz function 
for the $2 \times n$ generic matrix $X$ with $n \ge 2$~is
$$
\HK_{k[X]/I_2(X),\goth{m}}(q)
= {nq^{n+1} - (n-2)q^n \over 2} +n{n+ q-1 \choose n+1}. 
$$
\end{thm}

\proof
Recall that
$\HK_{k[X]/I_2(X),\goth{m}}(q)
= N_q(2,n;\overline{\infty}; \overline{\infty})$ by notation. 
We apply Theorems~\ref{thmN1} and \ref{thmrecur}:
\begin{align*}
N_q&(2,n;\overline{\infty}; \overline{\infty})
= 
N_q(2,n-1;\overline \infty; \overline \infty) \cr
&\hskip1em
+ \sum_{j=1}^{q-1}
(N_q(1,1;\infty;\infty ) \dotdiv N_q(1,1;\infty; q-1-j))
\cdot N_q(1,n-1;q-1- j;\overline \infty) \cr
&\hskip1em
+ \sum_{j=1}^{q-1}
N_q(1,1;\infty; q-1-j)
\cdot (N_q(1,n-1;\infty;\overline{q-1})
\dotdiv
N_q(1,n-1;\infty, q-1 - j;\overline{q-1})) \cr
&\hskip1em
+ \sum_{j=1}^{q-1}
N_q(0,1;; q-1-j)
\cdot (N_q(2,n-1;\overline \infty;\overline{q-1})
\dotdiv
N_q(2,n-1;\overline \infty, q-1 - j;\overline{q-1})) \cr
&\hskip1em
+ \sum_{j=1}^{q-1}
N_q(1,1;\infty; q-1-j)
\cdot N_q(1,n-1; \infty, q-1 - j; \overline \infty) \cr
&\hskip1em
+ \sum_{j=1}^{q-1}
N_q(0,1;; q-1-j)
\cdot N_q(2,n-1; \overline \infty, q-1 - j; \overline \infty) \cr
&= 
N_q(2,n-1;\overline \infty; \overline \infty)
+ \sum_{j=1}^{q-1}
N_q(1,1;\infty;\infty ) \cdot N_q(1,n-1;q-1- j;\overline \infty) \cr
&\hskip1em
+ \sum_{j=1}^{q-1}
N_q(1,1;\infty; q-1-j)
\cdot (N_q(1,n-1;\infty;\overline{q-1})
-
N_q(1,n-1;q-1 - j;\overline{q-1}))
\cr
&\hskip1em
+ \sum_{j=1}^{q-1}
(N_q(2,n-1;\overline \infty;\overline{q-1})
-
N_q(2,n-1;\infty, q-1 - j;\overline{q-1}))
+ \sum_{j=1}^{q-1}
N_q(2,n-1; \infty, q-1 - j; \overline \infty) \cr
&=
N_q(2,n-1;\overline \infty; \overline \infty)
+ \sum_{j=1}^{q-1} q \cdot {q - 1 - j + n - 1 \choose n-1}
+ \sum_{j=1}^{q-1}
(q-j) \left(q^{n-1} - {q-1-j+n-1 \choose n-1}\right)
\cr
&\hskip1em
+ \sum_{j=1}^{q-1}
\left(N_q(2,n-1;\overline \infty;\overline{q-1})
-
N_q(2,n-1;\infty, q-1 - j;\overline{q-1})
+ N_q(2,n-1; \infty, q-1 - j; \overline \infty)\right) \cr
&=
N_q(2,n-1;\overline \infty; \overline \infty)
+ q \cdot {q - 1 + n - 1 \choose n}
+ {q \choose 2} q^{n-1} - q \sum_{j=1}^{q-1} {q-1-j+n-1 \choose n-1}
+ \sum_{j=1}^{q-1} j {q-1-j+n-1 \choose n-1}
\cr
&\hskip1em
+ \sum_{j=1}^{q-1}
\left(N_q(2,n-1;\overline \infty;\overline{q-1})
-
N_q(2,n-1;\infty, q-1 - j;\overline{q-1})
+ N_q(2,n-1; \infty, q-1 - j; \overline \infty)\right) \cr
&=
N_q(2,n-1;\overline \infty; \overline \infty)
+ q \cdot {q - 1 + n - 1 \choose n}
+ {q \choose 2} q^{n-1} - q {q-1+n-1 \choose n}
+ {q - 1 + n \choose n+1} \hbox{ (by Lemma~\ref{lemsum})}
\cr
&\hskip1em
+ \sum_{j=1}^{q-1}
\left(N_q(2,n-1;\overline \infty;\overline{q-1})
-
N_q(2,n-1;\infty, q-1 - j;\overline{q-1})
+ 
N_q(2,n-1; \infty, q-1 - j; \overline \infty\right) \cr
&=
N_q(2,n-1;\overline \infty; \overline \infty)
+ {q \choose 2} q^{n-1}
+ {q - 1 + n \choose n+1}
\cr
&\hskip1em
+ \sum_{j=1}^{q-1}
\left(N_q(2,n-1;\overline \infty;\overline{q-1})
-
N_q(2,n-1;\infty, q-1 - j;\overline{q-1})
+ \sum_{j=1}^{q-1}
N_q(2,n-1; \infty, q-1 - j; \overline \infty\right). \cr
\end{align*}
When $n = 2$,
by Theorem~\ref{thmN1} this simplifies to
\begin{align*}
N_q&(2,2;\overline{\infty}; \overline{\infty})
=
q^2
+ {q \choose 2} q
+ {q + 1\choose 3}
+ \sum_{j=1}^{q-1}
\left(
{q-1 + 2 \choose 2} - \sum_{i = 0}^{q-1} \min\{q-i,q-j\}\right)
+ \sum_{j=1}^{q-1} q (q-j)
\cr
&= 
q^2
+ {q \choose 2} q
+ {q +1 \choose 3}
+ (q-1) { q+1 \choose 2}
- 2 {q+1 \choose 3}
\hbox{ (by Lemma \ref{lmq-iq-j})}
+ q {q \choose 2}
\cr
&= 
q^2
+ 2 {q \choose 2} q
+ (q-1) { q+1 \choose 2}
- 3 {q+1 \choose 3}
+ 2 {q +1 \choose 3}
\cr
&= 
{2q^2
+ 2 q^2 (q-1) \over 2}
+ {3 (q-1) (q+1) q
- 3 (q+1) q (q-1)q \over 6}
+ 2 {q +1 \choose 3}
\cr
&= 
{2q^3 \over 6}
+ 2 {q +1 \choose 3},
\cr
\end{align*}
which proves the case $n = 2$,
and if $n > 2$,
then by 
Theorems~\ref{thm2ninfq-1}, \ref{thm2ninfrinf}, \ref{thm2ninfrq-1},
by induction,
and the reduction above,
\begin{align*}
N_q&(2,n;\overline{\infty}; \overline{\infty})
=
{(n-1)q^n - (n-3)q^{n-1} \over 2} + (n-1) {n+q-2 \choose n}
+ {q \choose 2} q^{n-1} + {q-1 + n \choose n+1}
\cr
&\hskip1em
+ \sum_{j=1}^{q-1}
\left(
q^n + (n-3) q^{n-2} {q \choose 2}
-
(q-j)q^{n-1} + {q-1-j + n-1 \choose n}
+ (n-2) {q-1-j+n-1 \choose n} + (q-j) q^{n-1}\right)
\cr
&=
{(n-1)q^n - (n-3)q^{n-1} \over 2} + (n-1) {n+q-2 \choose n}
+ {q \choose 2} q^{n-1}
+ {q-1+ n \choose n+1}
\cr
&\hskip1em
+ (q-1) q^n
+ (q-1) (n-3) q^{n-2} {q \choose 2}
+ (n-1) {q+ n-2 \choose n+1}
\cr
&=
{(n-1)q^n - (n-3)q^{n-1} \over 2}
+ (q-1) q^n
+ (q-1) (n-3) q^{n-2} {q \choose 2}
+ {q \choose 2} q^{n-1}
+ (n-1) {n+q-1 \choose n+1}
+ {n+q-1 \choose n+1}
\cr
&=
{(n-1)q^n - (n-3)q^{n-1}
+ 2 q^{n+1}
- 2 q^n
+ (n-3) q^n (q-1)
- (n-3) q^{n-1} (q-1)
+ q^n (q-1) \over 2}
+ n {n+q-1 \choose n+1}
\cr
&=
{nq^{n-1}
- (n-2) q^n \over 2}
+ n {n+q-1 \choose n+1}
&\qedbox
\cr
\end{align*}

\begin{cor} \label{corHKmult}
(c.f.\ \cite[Theorem 3.3]{EY03})
The Hilbert-Kunz multiplicity
for the $2 \times n$ generic matrix $X$ with $n \ge 2$~is
$$
e_{\HK}(k[X]/I_2(X);\goth{m}) = \lim_{q \to \infty} {N_q(2,n;\overline \infty; \overline \infty) \over q^{n+1}}
= {n \over 2} + {n \over (n+1)!}.
\eqno\qedbox
$$
\end{cor}

We pause to use Theorem~\ref{thm:2byn} and Corollary~\ref{corHKmult} to record the 
Hilbert-Kunz functions and corresponding multiplicities 
for specific $n \ge 2$. We also remark that the Hilbert-Kunz function 
$N_q(2,3;\overline \infty; \overline \infty)$ appeared in \cite[pg. 544]{HMM04} without supporting calculation. 
\begin{align*}
&N_q(2,2;\overline \infty; \overline \infty) = {4q^3 - q \over 3}, 
\quad
e_{\HK} = {4 \over 3}, \cr
&
N_q(2,3;\overline \infty; \overline \infty) = {13 q^4 - 2q^3 - q^2 - 2q \over 8},
\quad
e_{\HK} = {13 \over 8}, \cr
&
N_q(2,4;\overline \infty; \overline \infty) = {61 q^5 - 25q^4 + 5 q^3 - 5 q^2 - 6 q \over 30},
\quad
e_{\HK} = {61 \over 30}, \cr
&
N_q(2,5;\overline \infty; \overline \infty) = {361 q^6 - 207 q^5 + 25q^4 + 15 q^3 - 26 q^2 - 24 q \over 144},
\quad
e_{\HK} = {361 \over 144}.
\end{align*}

To generalize Theorem~\ref{thm:2byn}, the recursion in Theorem~\ref{thmrecur}
applied to larger cases forces other cases of $N_q(2,n; r_1, r_2; c_1, \ldots, c_n)$ to be computed.
The rest of this section proves two such necessary cases.

\begin{thm}\label{thm2q-1rq-1}
For all $r < q$,
$N_q(2,n;q-1,r; \overline{q-1})
= 
q{r+n \choose n}
+ \sum_{i=1}^{n-1}
{q+i-1 \choose i+1} {r+n-i \choose n-i}
- n {r+n \choose n+1}$.
\end{thm}

\proof
By Theorem~\ref{thmrecur},
\begin{align*}
N_q&(2,n;q-1,r; \overline {q-1})
= 
N_q(2,n-1;q-1,r; \overline {q-1}) \cr
&\hskip2em
+ \sum_{j=1}^{q-1}
(N_q(1,1;r; q-1-j) \dotdiv N_q(1,1;r; q-1-j))
\cdot N_q(1,n-1;q-1- j;\overline {q-1}) \cr
&\hskip2em
+ \sum_{i=1}^2 \sum_{j=1}^{r_i}
N_q(2-i,1;\overline{q-1}, r; q-1-j) \cr
&\hskip6em
\cdot (N_q(i,n-1;r_1, \ldots, r_{i-1}, r_i - j;\overline {q-1})
\dotdiv
N_q(i,n-1;r_1, \ldots, r_{i-1}, q-1 - j;\overline {q-1})) \cr
&\hskip2em
+ \sum_{i=1}^2 \sum_{j=1}^{r_i}
N_q(2-i,1;\overline{q-1}, r; q-1-j)
\cdot N_q(i,n-1;r_1, \ldots, r_{i-1}, r_i - j;\overline {q-1}) \cr
&= 
N_q(2,n-1;q-1,r; \overline {q-1}) \cr
&\hskip2em
+ \sum_{i=1}^2 \sum_{j=1}^{r_i}
N_q(2-i,1;\overline{q-1}, r; q-1-j)
\cdot N_q(i,n-1;r_1, \ldots, r_{i-1}, r_i - j;\overline {q-1}) \cr
&= 
N_q(2,n-1;q-1,r; \overline {q-1}) \cr
&\hskip2em
+ \sum_{j=1}^{q-1}
N_q(1,1;r; q-1-j) \cdot N_q(1,n-1;q-1 - j;\overline {q-1}) \cr
&\hskip2em
+ \sum_{j=1}^r
N_q(0,1;; q-1-j)
\cdot N_q(2,n-1;q-1, r - j;\overline {q-1}) \cr
&= 
N_q(2,n-1;q-1,r; \overline {q-1}) \cr
&\hskip2em
+ \sum_{j=1}^{q-1}
\min\{r+1,q-j\} {q-1-j+n-1 \choose n-1}
+ \sum_{j=1}^r
N_q(2,n-1;q-1, r - j;\overline {q-1}) \cr
&= 
\sum_{j=0}^r N_q(2,n-1;q-1, r - j;\overline {q-1})
+ (r+1) {q+n-2 \choose n}
- {r + n \choose n+1}
\hbox{ (by Lemma~\ref{lmminr+1q-jbinom})}.
\cr
\end{align*}
When $n = 2$,
this equals
\begin{align*}
N_q&(2,2;q-1,r; \overline {q-1})
= 
\sum_{j=0}^r N_q(2,1;q-1, r - j;\overline {q-1})
+ (r+1) {q \choose 2} - {r+2 \choose 3}
\cr
&= 
\sum_{j=0}^r \sum_{i=0}^{q-1} \min\{q-i,r-j+1\} 
+ (r+1) {q \choose 2} - {r+2 \choose 3}
\cr
&= 
q {r + 2 \choose 2} - {r + 2 \choose 3} \hbox{ (by Lemma~\ref{lmothermin}) }
+ (r+1) {q \choose 2} - {r+2 \choose 3}
\cr
&= 
q {r + 2 \choose 2}
+ (r+1) {q \choose 2}
- 2 {r + 2 \choose 3}, 
\cr
\end{align*}
which fits the pattern,
and for $n > 2$,
\begin{align*}
N_q&(2,n;q-1,r; \overline {q-1})
= 
\sum_{j=0}^r N_q(2,n-1;q-1, r - j;\overline {q-1})
+ (r+1) {q+n-2 \choose n}
- {r + n \choose n+1}
\cr
&= 
\sum_{j=0}^r
\left(
q {r-j+n-1 \choose n-1}
+ \sum_{i=1}^{n-2}
{q + i-1 \choose i+1} {r-j+n-1-i \choose n-1-i}
- (n-1) {r-j + n-1 \choose n}
\right) \cr
&\hskip6em
+ (r+1) {q+n-2 \choose n}
- {r + n \choose n+1}
\cr
&= 
q {r+n \choose n}
+ \sum_{i=1}^{n-2} {q + i-1 \choose i+1} {r+n-i \choose n-i}
- (n-1) {r + n \choose n+1}
+ (r+1) {q+n-2 \choose n}
- {r + n \choose n+1}
\cr
&= 
q {r+n \choose n}
+ \sum_{i=1}^{n-1} {q + i-1 \choose i+1} {r+n-i \choose n-i}
- n {r + n \choose n+1}. &\qedbox
\cr
\end{align*}

\begin{thm}\label{thm2q-1rinfty}
For all $r < q$,
$N_q(2,n;q-1,r; \overline{\infty})
= q {r+n \choose n}
+ \sum_{i=1}^{n-1} {q+i-1 \choose i+1} {r+n-i \choose n-i}$.
\end{thm}

\proof
By Theorem~\ref{thmrecur},
\begin{align*}
N_q&(2,n;q-1,r; \overline {\infty})
= 
N_q(2,n-1;q-1,r; \overline {\infty}) \cr
&\hskip2em
+ \sum_{j=1}^{q-1}
(N_q(1,1;r; \infty) \dotdiv N_q(1,1;r; q-1-j))
\cdot N_q(1,n-1;q-1- j;\overline {\infty}) \cr
&\hskip2em
+ \sum_{i=1}^2 \sum_{j=1}^{r_i}
N_q(2-i,1;\overline{q-1}, r; q-1-j) \cr
&\hskip6em
\cdot (N_q(i,n-1;r_1, \ldots, r_{i-1}, r_i - j;\overline {q-1})
\dotdiv
N_q(i,n-1;r_1, \ldots, r_{i-1}, q-1 - j;\overline {q-1})) \cr
&\hskip2em
+ \sum_{i=1}^2 \sum_{j=1}^{r_i}
N_q(2-i,1;\overline{q-1}, r; q-1-j)
\cdot N_q(i,n-1;r_1, \ldots, r_{i-1}, r_i - j;\overline {\infty}) \cr
&= 
N_q(2,n-1;q-1,r; \overline {\infty}) \cr
&\hskip2em
+ \sum_{j=1}^{q-1}
(N_q(1,1;r; \infty) - N_q(1,1;r; q-1-j))
\cdot N_q(1,n-1;q-1- j;\overline {\infty}) \cr
&\hskip2em
+ \sum_{j=1}^{q-1}
N_q(1,1;r; q-1-j)
\cdot N_q(1,n-1;q-1- j;\overline {\infty}) \cr
&\hskip2em
+ \sum_{j=1}^r
N_q(0,1;; q-1-j)
\cdot N_q(2,n-1;q-1,r - j;\overline {\infty}) \cr
&= 
\sum_{j=0}^r
N_q(2,n-1;q-1,r - j;\overline {\infty})
+ \sum_{j=1}^{q-1}
N_q(1,1;r; \infty) 
\cdot N_q(1,n-1;q-1- j;\overline {\infty}) \cr
&= 
\sum_{j=0}^r
N_q(2,n-1;q-1,r - j;\overline {\infty})
+ \sum_{j=1}^{q-1}
(r+1) \cdot {q-1-j + n-1 \choose n-1} \cr
&= 
\sum_{j=0}^r
N_q(2,n-1;q-1,r - j;\overline {\infty})
+ (r+1) \cdot {q+ n-2 \choose n}. \cr
\end{align*}
When $n = 2$,
this equals
\begin{align*}
N_q&(2,2;q-1,r; \overline {\infty})
= 
\sum_{j=0}^r N_q(2,1;q-1,r - j;\overline {\infty})
+ (r+1) \cdot {q \choose 2} \cr
&= 
q \sum_{j=0}^r (r-j+1)
+ (r+1) \cdot {q \choose 2}
= q {r+2 \choose 2} + (r+1) \cdot {q \choose 2}, \cr
\end{align*}
which fits the pattern,
and for $n > 2$,
\begin{align*}
N_q&(2,n;q-1,r; \overline {\infty})
= 
\sum_{j=0}^r N_q(2,n-1;q-1, r - j;\overline {\infty})
+ (r+1) {q+n-2 \choose n}
\cr
&= 
\sum_{j=0}^r
\left(
q {r-j+n-1 \choose n-1}
+ \sum_{i=1}^{n-2}
{q + i-1 \choose i+1} {r-j+n-1-i \choose n-1-i}
\right)
+ (r+1) {q+n-2 \choose n}
\cr
&= 
q {r+n \choose n}
+ \sum_{i=1}^{n-2} {q + i-1 \choose i+1} {r+n-i \choose n-i}
+ (r+1) {q+n-2 \choose n}
\cr
&= 
q {r+n \choose n}
+ \sum_{i=1}^{n-1} {q + i-1 \choose i+1} {r+n-i \choose n-i}. &\qedbox
\cr
\end{align*}

%
%


\begin{thm}\label{thmNFsign}
For all $r < q$,
$N_q(2,n;\infty,r; \infty, \overline{q-1}) = q^n (r+1)$.
\end{thm}

\proof
By Theorem~\ref{thmrecur},
\begin{align*}
N_q&(2,n;\infty,r; \infty, \overline{q-1})
= 
N_q(2,n-1;\infty,r; \overline{q-1}) \cr
&
+ \sum_{j=1}^{q-1}
(N_q(1,1;r; \infty) \dotdiv N_q(1,1;r; q-1-j))
\cdot N_q(1,n-1;q-1-j; \overline{q-1}) \cr
&
+ \sum_{j=1}^{q-1}
N_q(1,1;r; q-1-j)
\cdot (N_q(1,n-1;\infty;\overline{q-1})
\dotdiv
N_q(1,n-1;q-1 - j;\overline{q-1})) \cr
&
+ \sum_{j=1}^{r}
N_q(0,1;; q-1-j)
\cdot (N_q(2,n-1;\infty, r - j;\overline{q-1})
\dotdiv
N_q(2,n-1;\infty, q-1 - j;\overline{q-1})) \cr
&
+ \sum_{j=1}^{q-1}
N_q(1,1;r; q-1-j) \cdot N_q(1,n-1;q-1 - j;\overline{q-1}) \cr
&
+ \sum_{j=1}^{r}
N_q(0,1;; q-1-j) \cdot N_q(2,n-1;\infty, r - j;\overline{q-1}) \cr
&= 
\sum_{j=0}^{r} N_q(2,n-1;\infty, r - j;\overline{q-1}) \cr
&
+ \sum_{j=1}^{q-1}
N_q(1,1;r; \infty) \cdot N_q(1,n-1;q-1-j; \overline{q-1}) \cr
&
+ \sum_{j=1}^{q-1}
N_q(1,1;r; q-1-j)
\cdot (N_q(1,n-1;\infty;\overline{q-1})
-
N_q(1,n-1;q-1 - j;\overline{q-1})) \cr
&= 
\sum_{j=0}^{r} N_q(2,n-1;\infty, r - j;\overline{q-1})
+ \sum_{j=1}^{q-1}
(r+1) \cdot {q-1-j + n-1 \choose n-1} \cr
&
+ \sum_{j=1}^{q-1}
\min\{r+1,q-j\} 
\cdot \left(q^{n-1}
-
 {q-1-j + n-1 \choose n-1}\right)
\hbox{  (by Theorem~\ref{thmN1})} \cr
&= 
\sum_{j=0}^{r} N_q(2,n-1;\infty, r - j;\overline{q-1})
+ (r+1) \cdot {q-2 + n \choose n} \cr
&
+ 
q^{n-1} \left( q (r+1) - {r+2 \choose 2} \right)
-
(r+1) {q+n-2 \choose n} + {r+n \choose n+1}
\hbox{  (by Lemma~\ref{lmminr+1q-jbinom})} \cr
&= 
\sum_{j=0}^{r} N_q(2,n-1;\infty, r - j;\overline{q-1})
+ q^n (r+1) - q^{n-1} {r+2 \choose 2} + {r+n \choose n+1}.
\cr
\end{align*}
When $n = 2$,
by Theorem~\ref{thmN1} and Lemma~\ref{lmothermin},
this equals
\begin{align*}
N_q&(2,2;\infty,r; \infty, q-1)
= 
\sum_{j=0}^{r} \sum_{i=0}^{q-1} \min\{q-i, r-j+1\}
+ q^2 (r+1) - q {r+2 \choose 2} + {r+2 \choose 3}
\cr
&= 
q{r+2 \choose 2} - {r+2 \choose 3}
+ q^2 (r+1) - q {r+2 \choose 2} + {r+2 \choose 3}
\cr
&= 
q^2 (r+1),
\cr
\end{align*}
and when $n > 2$,
by Theorem~\ref{thm2ninfrq-1}
this equals
\begin{align*}
N_q&(2,2;\infty,r; \infty, q-1)
= 
\sum_{j=0}^{r} \left((r-j+1) q^{n-1} - {r-j+n-1 \choose n}\right)
+ q^n (r+1) - q^{n-1} {r+2 \choose 2} + {r+n \choose n+1}
\cr
&= 
{r+2 \choose 2} q^{n-1} - {r+n \choose n+1}
+ q^n (r+1) - q^{n-1} {r+2 \choose 2} + {r+n \choose n+1}
\cr
&= 
q^n (r+1),
\cr
\end{align*}
which proves the theorem.
\qed

\appendix
\section{Some binomial formulas}\label{sec:binform}

This section contains some binomial formulas used in the main calculations.
The reader may want to read this section only as needed.

\begin{lem}\label{lemsum}
For any $q\ge 1$,
$$
\sum_{j=0}^{q-1} j {j+n-1 \choose n-1}
= n {q+n-1 \choose n+1}
\hbox{ and }
\sum_{j=1}^q j {q-j+n-1 \choose n-1}
= {q+n \choose n+1}.
$$
\end{lem}

\proof Note
$\sum_{j=0}^{q-1} j {j+n-1 \choose n-1}
= \sum_{j=0}^{q-1} j {(j+n-1)! \over (n-1)! j!}
= n \sum_{j=0}^{q-1} {(j+n-1)! \over n! (j-1)!}
= n \sum_{j=0}^{q-1} {j+n-1 \choose n}
= n {q+n-1 \choose n+1}$,
so that
\begin{align*}
\sum_{j=1}^q j {q-j+n-1 \choose n-1}
&= \sum_{k=0}^{q-1} (q-k) \cdot {k+n-1 \choose n-1} \cr
&= q \sum_{j=0}^{q-1} {j+n-1 \choose n-1}
- \sum_{j=0}^{q-1} j {j+n-1 \choose n-1} \cr
&= q {q+n-1 \choose n} - n {q+n-1 \choose n+1} \cr
&= q {(q-1+n)! \over n! (q-1)!}
- n {(q-1+n)! \over (n+1)! (q-2)!} \cr
&= q (n+1) {(q-1+n)! \over (n+1)! (q-1)!}
- n (q-1) {(q-1+n)! \over (n+1)! (q-1)!} \cr
&= (n+q) {(q-1+n)! \over (n+1)! (q-1)!} = {q + n \choose n + 1}. &\qedbox \cr
\end{align*}

\begin{lem}\label{lmminr+1q-jbinom}
For all $r < q$, we have the following equalities
\begin{align*}
\sum_{j=1}^{q-1} \min\{r+1,q-j\}
&=
(q-1)(r+1) - {r+1 \choose 2}
= q (r+1) - {r+2 \choose 2},
\cr
\sum_{j=1}^{q-1} \min\{r+1,q-j\} &\cdot {q-1-j+n-1 \choose n-1}
= (r+1) {q+n-2 \choose n} - {r + n \choose n+1},
\cr
\sum_{j=1}^{q-1} \min\{r+1,q-j\} &\cdot {q-1-j+n-1 \choose n}
\cr
&= 
(r+1) {q+n-2 \choose n+1} - (r-1) {r+n-1 \choose n+1}
+ (n+1) {r+n-1 \choose n+2}.
\cr
\end{align*}
\end{lem}

\proof By direct computation
\begin{align*}
\sum_{j=1}^{q-1}
&\min\{r+1,q-j\} =
\sum_{j=1}^{q-r-1} (r+1) + \sum_{j=q-r}^{q-1} (q-j)
= (q-r-1)(r+1) + {r+1 \choose 2} \cr
&= (q-1)(r+1) - {r+1 \choose 2},
\cr
\sum_{j=1}^{q-1}
&\min\{r+1,q-j\} \cdot {q-1-j+n-1 \choose n-1} \cr
&=
\sum_{j=1}^{q-r-1}
(r+1) \cdot {q-1-j+n-1 \choose n-1}
+ \sum_{j=q-r}^{q-1}
(q-j) \cdot {q-1-j+n-1 \choose n-1} \cr
&= 
(r+1) \left({q+n-2 \choose n} - {r+n-1 \choose n}\right)
\cr
&\hskip4em
+ \sum_{j=q-r}^{q-1} (q-1-j) {q-1-j+n-1 \choose n-1}
+ \sum_{j=q-r}^{q-1} {q-1-j+n-1 \choose n-1}
\cr
&= 
(r+1) \left({q+n-2 \choose n} - {r+n-1 \choose n}\right)
+ \sum_{k=0}^{r-1} k {k+n-1 \choose n-1}
+ {r+n-1 \choose n}
\cr
&=
(r+1) {q+n-2 \choose n} - r  {r+n-1 \choose n}
+ n {r+n-1 \choose n+1}
\hbox{ (by Lemma~\ref{lemsum})}
\cr
&=
(r+1) {q+n-2 \choose n}
- {r + n \choose n+1},
\cr
\sum_{j=1}^{q-1} &\min\{r+1,q-j\} \cdot {q-1-j+n-1 \choose n}
\cr
&=
\sum_{j=1}^{q-r-1}
(r+1) \cdot {q-1-j+n-1 \choose n}
+ \sum_{j=q-r}^{q-1}
(q-j) \cdot {q-1-j+n-1 \choose n}
\cr
&=
(r+1) \left(
{q+n-2 \choose n+1}
- {r+n-1 \choose n+1}
\right)
+ \sum_{j=q-r}^{q-2}
(q-j) \cdot {q-2-j+n \choose n}
\cr
\cr
&=
(r+1) {q+n-2 \choose n+1} - (r+1) {r+n-1 \choose n+1}
\cr
&\hskip2em
+ \sum_{j=q-r}^{q-2}
(q-2-j) \cdot {q-2-j+n \choose n}
+2 \sum_{j=q-r}^{q-2} {q-2-j+n \choose n}
\cr
&=
(r+1) {q+n-2 \choose n+1} - (r+1) {r+n-1 \choose n+1}
+ \sum_{k=0}^{r-2}
k \cdot {k+n \choose n}
+2 {r-1+n \choose n+1}
\cr
&=
(r+1) {q+n-2 \choose n+1} - (r-1) {r+n-1 \choose n+1}
+ (n+1) {r+n-1 \choose n+2} \hbox{ (by Lemma~\ref{lemsum})}.
&\qedbox \cr
\end{align*}

\begin{lem}\label{lmq-iq-j}
For all integers $q \ge 0$,
$\displaystyle \sum_{j=1}^{q-1} \sum_{i=0}^{q-1} \min\{q-i,q-j\}
= 2 {q +1 \choose 3}$.
\end{lem}

\begin{proof} By direct computation
$\sum_{j=1}^{q-1} \sum_{i=0}^{q-1} \min\{q-i,q-j\}
= 
\sum_{j=1}^{q-1} \sum_{i=0}^{j-1} (q-j)
+ \sum_{j=1}^{q-1} \sum_{i=j}^{q-1} (q-i)
= 
\sum_{j=1}^{q-1} j {q-1-j+1 \choose 1}
+ \sum_{j=1}^{q-1} {q-j+1 \choose 2}
= 
{q+1 \choose 3} + {q+1 \choose 3}$.
\end{proof}

\begin{lem}\label{lmothermin}
For all $r < q$,
$\sum_{i=0}^r \sum_{j=1}^{q-1} \min\{r-i+1,q-j\}
= q {r+2 \choose 2} - {r+3 \choose 3}$,
and $\sum_{i=0}^r \sum_{j=0}^{q-1} \min\{r-i+1,q-j\}
= q {r+2 \choose 2} - {r+2 \choose 3}.$
\end{lem}

\proof
By Lemma~\ref{lmminr+1q-jbinom},
\begin{align*}
\sum_{i=0}^r &\sum_{j=1}^{q-1} \min\{r-i+1,q-j\}
= \sum_{i=0}^r \left( q (r-i+1) - {r-i+2 \choose 2}\right)
= q {r+2 \choose 2} - {r+3 \choose 3},
\cr
\sum_{i=0}^r &\sum_{j=0}^{q-1} \min\{r-i+1,q-j\}
= q {r+2 \choose 2} - {r+3 \choose 3} + \sum_{i=0}^r (r-i+1) \cr
&= q {r+2 \choose 2} - {r+3 \choose 3} + {r+2 \choose 2}
= q {r+2 \choose 2} - {r+2 \choose 3}.
& \qedbox
\cr
\end{align*}

\end{document}